\documentclass[a4paper,11pt]{llncs}
\usepackage{lmodern}
\usepackage{amsfonts,mathrsfs,amsmath,graphicx,amssymb}

\title{Derivatives of Multilinear Functions of Matrices}
\author{Priyanka Grover}
\institute{{Theoretical Statistics and Mathematics Unit,}\\\small{Indian Statistical Institute, Delhi Centre,}\\ \small{7, S.J.S. Sansanwal Marg,}\\ \small{New Delhi-110016, India}\\
\email{pgrover8r@isid.ac.in}}

\newcommand{\X}{\mathbb X}
\newcommand{\Y}{\mathbb Y}
\newcommand{\N}{\mathbb N}
\newcommand{\mat}{\mathbb{M}(n)}
\newcommand{\matk}{\mathbb{M}(\binom{n}{k})}

\newcommand{\C}{\mathbb{C}}
\newcommand{\I}{\mathcal{I}}
\newcommand{\J}{\mathcal{J}}

\newcommand{\tr}{\mathop{{\rm tr}}}
\newcommand{\padj}{\mathop{{\rm padj}}}
\newcommand{\adj}{\mathop{{\rm adj}}}
\newcommand{\sgn}{\mathop{{\rm sgn}}}
\newcommand{\per}{\mathop{{\rm per}}}
\newtheorem{thm}{Theorem}

\newtheorem{cor}[thm]{Corollary}

\newcommand{\Hil}{\mathcal{H}}
\newcommand{\lo}{\mathscr{L}}
\newcommand{\De}{{\rm D}}

\begin{document}
\index{Grover, Priyanka}
\maketitle

Perturbation or error bounds of functions have been of great interest for a long time. If the functions are differentiable, then the mean value theorem and Taylor's theorem come handy for this purpose. While the former is useful in estimating $\|f(A+X)-f(A)\|$ in terms of $\|X\|$ and requires the norms of the first derivative of the function, the latter is useful in computing higher order perturbation bounds and needs norms of the higher order derivatives of the function.\\

In the study of matrices, determinant is an important function. Other scalar valued functions like eigenvalues and coefficients of characteristic polynomial are also well studied. Another interesting function of this category is the permanent, which is an analogue of the determinant in matrix theory. More generally, there are operator valued functions like tensor powers, antisymmetric tensor powers and symmetric tensor powers which have gained importance in the past. In this article, we give a survey of the recent work on the higher order derivatives of these functions and their norms. Using Taylor's theorem, higher order perturbation bounds are obtained. Some of these results are very recent and their detailed proofs will appear elsewhere.

\numberwithin{thm}{section}
\numberwithin{equation}{section}

\section{Introduction}
Let $\X$ and $\Y$ be two Banach spaces. Let $f : \X \rightarrow \Y$ be a continuously differentiable map. The derivative of $f$ at a point $a \in \X$ is the linear map $\De f(a):\X \rightarrow \Y$ whose action at $x \in \X$ is given by :--
\begin{equation}
\De f{(a)}(x)=\left.\frac{d}{dt}\right|_{t=0} f{(a+tx)}. \label{3} 
\end{equation}
Let $\lo (\mathbb X;\mathbb Y)$ denote the Banach space of all bounded linear operators from $\X$ into $\Y$. The map $\De f$ is a continuous map of $\X$ into $\lo (\X;\Y)$. If $\De f$ is differentiable at $a \in \X$, then $f$ is said to be twice differentiable at $a$, and the $\emph{second derivative}$ of $f$ at $a$, denoted by $\De^2 f{(a)}$, is the derivative of $\De f$ at $a$. This is an element of $\lo(\X;(\lo(\X;\Y))$ which is identified with $\lo_2(\X;\Y)$, the space of continuous bilinear mappings of $\X \times \X$ into $\Y$. Similarly, for any $m$, if $\De ^{m-1} f$ is differentiable at $a\in \X$, then $f$ is said to be $m$-times differentiable at $a$. The $m^{th}$ derivative of $f$ at $a$, denoted by $\De^m f{(a)}$, is an element of $\mathscr{L}_m(\X;\Y)$, the space of continuous multilinear mappings of $\X \times \cdots \times \X$ into $\Y$ (See \cite{dieudonne}). If $f$ is $m$-times differentiable, then for $x^1,\ldots,x^m \in \X$
\begin{equation}
\De^m f{(a)(x^1,\ldots,x^m)}=\left.\frac{\partial^m}{\partial t_1 \cdots  \partial t_m} \right|_{t_1=\cdots=t_m=0} f{(a+t_1 x^1+\cdots+t_m x^m)}.\label{derivativedefn}
\end{equation}

The norm of a linear operator $T$ is defined as
$$\|T\|=\sup_{\|x\|=1}\|T(x)\|.$$
It follows that
\begin{equation}
\|\De^m f(a)\|=\sup_{\|x^1\|=\cdots=\|x^m\|=1} \|\De ^m f(a)(x^1,\ldots,x^m)\|.\label{normhighder}
\end{equation}
Taylor's theorem says that if $f$ is a (p+1)-times differentiable function, then for all $a\in \X$ and for small $x\in \X$
\begin{equation}
\|f(a+x)-f(a)\|\leq \sum_{m=1}^{p} \frac{1}{m!}\|\De^m f(a)\| \|x\|^m+O(\|x\|^{p+1}).\label{taylor}
\end{equation}
In order to find higher order perturbation bounds, one needs to know the norms of $\De^m f(a)$ or upper bounds on them, for all $m$. In this article we discuss these for some important multilinear functions of matrices. 

Let $\Hil$ be an $n$-dimensional complex Hilbert space with the usual Euclidean norm $\|\cdot\|$. We identify $\Hil$ with $\C^n$ and the set $\lo(\Hil,\Hil)$ with the set $\mat$ of all $n\times n$ complex matrices. For $A\in \mat$ the operator norm of $A$ is defined as
$$\|A\|=\sup_{x\in \C^n, \ \|x\|=1}\|Ax\|.$$ Let $s_1(A)\geq \cdots\geq s_n(A)\geq 0$ be the singular values of $A$. Then
$$\|A\|=s_1(A).$$

Let $f:\mathbb M(n_1)\rightarrow \mathbb M(n_2)$ be an $m$-times differentiable map. Then the norm of $\De^m f(A)$ is given by
\begin{equation}
\|\De^m f(A)\|=\sup_{\|X^1\|=\cdots=\|X^m\|=1} \|\De^m f(A)(X^1,\ldots,X^m)\|.
\end{equation}

Let $\otimes ^k \Hil$ denote the $k$-fold tensor power of $\Hil$. It is a Hilbert space of dimension $n^k$ (See \cite[Chapter 1]{R.bhatia3}). If $\{e_i\}$, $1\leq i\leq n$, is an orthonormal basis of $\Hil$, then $\{e_{i_1}\otimes\cdots\otimes e_{i_k}:1\leq i_1,\dots,i_k\leq n\}$ forms a basis for $\otimes ^k \Hil$. We order this basis lexicographically. Let $\langle \cdot, \cdot \rangle $ denote the inner product on $\Hil$. Then the inner product in $\otimes ^k \Hil$ is defined by
$$\langle x_1\otimes\cdots\otimes x_k,y_1\otimes\cdots\otimes y_k\rangle=\prod_{i=1}^k \langle x_i,y_i\rangle.$$ 
The \emph{tensor power} of $A$, denoted by $\otimes^k A$, is a map from the space $\mat$ to $\mathbb M(n^k)$. It is defined on elementary tensors by
$$(\otimes^k A)(x_1\otimes\cdots\otimes x_k)=Ax_1\otimes \cdots\otimes Ax_k,$$ and then extended linearly to all of $\otimes^k \Hil$.

Two important subspaces of $\otimes ^k \Hil$ are the \emph{antisymmetric tensor power} and the \emph{symmetric tensor power} of $\Hil$. The antisymmetric tensor product of vectors $x_1,\ldots,x_k$ in $\Hil$ is defined as 
$$x_1\wedge\cdots\wedge x_k=\frac{1}{(k!)^{1/2}} \sum_{\sigma \in S_k} \sgn(\sigma) x_{\sigma(1)}\otimes \cdots \otimes x_{\sigma(k)},$$
where $S_k$ denotes the set of all permutations of $\{1,2,\ldots,k\}$ and $\sgn(\sigma)=\pm 1$, depending on whether $\sigma$ is an even or odd permutation. If $x_j$ are orthonormal, then $x_1\wedge\cdots\wedge x_k$ is a unit vector. Note that
$$x_1\wedge \cdots \wedge x_i \wedge \cdots \wedge x_j\wedge \cdots\wedge x_k=-x_1\wedge \cdots \wedge x_j \wedge \cdots\wedge x_i\wedge \cdots\wedge x_k.$$ In particular, $x_1\wedge \cdots\wedge x_k=0$ if $x_i=x_j$ for some $i\neq j$. The span of all antisymmetric tensors $x_1\wedge \cdots \wedge x_k$ in $\otimes^k \Hil$ is called the \emph{antisymmetric tensor power} of $\Hil$. It is denoted by $\wedge ^k \Hil$. For $k>n$ the space $\wedge^k \Hil=\{0\}$. Let
$Q_{k,n}=\{(i_1,\ldots,i_k) |\ i_1,\ldots,i_k\in \N, \ 1\leq i_1<\cdots<i_k\leq n\}.$ For $k>n,\, Q_{k,n}=\varnothing$, by convention.
If $\{e_i\}$, $1\leq i\leq n$, is an orthonormal basis of $\Hil$, then for $\alpha=(\alpha_1,\ldots,\alpha_k) \in Q_{k,n}$ we define $$e^{(\alpha)}=e_{\alpha_1}\wedge \cdots \wedge e_{\alpha_k}.$$ Then $\{ e^{(\alpha)}: \alpha \in Q_{k,n}\}$ forms an orthonormal basis of $\wedge^k \Hil$. The restriction of $\otimes^k A$ to this subspace is denoted by $\wedge^k A$ and is called the $k$th \emph{antisymmetric tensor power} of $A$.  Given two elements $\alpha$ and $\beta$ of $Q_{k,n}$, let $A[\alpha|\beta]$ denote the $k \times k$ matrix obtained from $A$ by picking its entries from the rows corresponding to $\alpha$ and the columns corresponding to $\beta$. With respect to the above mentioned basis, the $(\alpha,\beta)$-entry of $\wedge^k A$ is $\det{A[\alpha|\beta]}$. 

The symmetric tensor product of vectors $x_1,\ldots,x_k$ in $\Hil$ is defined as 
$$x_1\vee\cdots\vee x_k=\frac{1}{(k!)^{1/2}} \sum_{\sigma \in S_k} x_{\sigma(1)}\otimes \cdots \otimes x_{\sigma(k)}.$$ If $x_j$ are orthonormal, then $x_1\vee\cdots\vee x_k$ is a unit vector. The span of all these vectors in $\otimes^k \Hil$ is denoted by $\vee ^k \Hil$. It is called the \emph{symmetric tensor power} of $\Hil$. Let
$G_{k,n}=\{(i_1,\ldots,i_k)|\ i_1,\ldots,i_k\in \N, 1\leq i_1\leq\cdots\leq i_k\leq n\}.$ 
Note here that for $k\leq n,\ Q_{k,n}$ is a subset of $G_{k,n}$. Given an orthonormal basis $\{e_i\}$, $1\leq i\leq n$, of $\Hil$ define, for $\alpha=(\alpha_1,\ldots,\alpha_k) \in G_{k,n}$ $$e_{(\alpha)}=e_{\alpha_1}\vee \cdots \vee e_{\alpha_k}.$$ If $\alpha$ consists of $\ell$ distinct indices $\alpha_1,\ldots,\alpha_{\ell}$ with multiplicities $m_1,\ldots,m_{\ell}$ respectively, put $m(\alpha)=m_1!\cdots m_{\ell}!$. Note that if $\alpha \in Q_{k,n}$, then $m(\alpha)= 1$. The set $\{m(\alpha)^{-1/2}\ e_{(\alpha)}: \alpha \in G_{k,n}\}$ is an orthonormal basis of $\vee^k \Hil$. The restriction of $\otimes^k A$ to this subspace is denoted by $\vee^k A$ and is called the $k$th \emph{symmetric tensor power} of $A$.  Given two elements $\alpha$ and $\beta$ of $G_{k,n}$, let $A[\alpha|\beta]$ denote the $k \times k$ matrix whose $(r,s)$-entry is the $(i_r,j_s)$-entry of $A$. Note that it may not be a submatrix of $A$. With respect to the above mentioned basis, the $(\alpha,\beta)$-entry of $\vee^k A$ is $(m(\alpha)m(\beta))^{-1/2} \per{A[\alpha|\beta]}$. 

There is a classical formula due to Jacobi for the first derivative of the determinant map. Bhatia and Jain \cite{tanvi} obtained expressions for higher order derivatives of this map. These are discussed in Section 2. Along with the determinant the permanent function has been of great interest. In Section 3 we give formulas for derivatives of all orders for the permanent function that we obtained in \cite{grover}. Then we move on to operator valued functions. We give formulas for the derivatives (of all orders) for the maps that take a matrix $A$ to $\otimes^k A$, $\wedge^k A$ and $\vee^k A$, in Sections 4, 5 and 6 respectively. The formulas for the derivatives of the map $A\rightarrow \wedge^k A$ were obtained by Jain \cite{tanvi2}. Note that $\wedge^n A=\det A$. The derivatives of the map $A\rightarrow \vee^k A$ were obtained by us in \cite{grover}. In another direction, the determinant of $A$ is the $n$th coefficient of the characteristic polynomial of an $n\times n$ matrix $A$. In Section 7 we discuss all coefficients in the characteristic polynomial and their derivatives, obtained in \cite{tanvi2}. A major application of these formulas is to find perturbation bounds for these functions and we show how these are obtained.

To state the results concisely, we need some multiindex notations which we briefly recall from \cite{R.bhatia3}, \cite{tanvi} and \cite{grover}.

\textbf{Notations.}
For $\I=(i_1,\ldots,i_k)$, the symbol $|\I|$ denotes the sum $i_1+\cdots+i_k$.
If $\I,\J \in Q_{m,n}$, then we denote by $A(\I|\J)$, the $(n-m)\times(n-m)$ submatrix obtained from $A$ by deleting rows corresponding to $\I$ and columns corresponding to $\J$. The $j^{th}$ column of a matrix $X$ is denoted by $X_{[j]}$. Given $n \times n$ matrices $X^1$,\ldots,$X^m$ and $\J=(j_1,\ldots,j_m) \in Q_{m,n}$, we denote by $A(\J;X^1,\ldots,X^m)$, the matrix obtained from $A$ by replacing the $j_p^{th}$ column of $A$ by the $j_p^{th}$ column of $X^p$ for $1\leq p \leq m$, and keeping the rest of the columns unchanged, that is, if $Z= A(\J;X^1,\ldots,X^m)$, then $Z_{[j_p]}=X^p_{[j_p]}$ for $1\leq p \leq m$, and 
$Z_{[\ell]}=A_{[\ell]}$ if $\ell$ does not occur in $\J$. Let $\sigma$ be a permutation on $m$ symbols, then $Y^{\sigma}_{[\J]}$ denotes the matrix in which $Y^{\sigma}_{[j_p]}=X^{\sigma(p)}_{[j_p]}$ for $1\leq p \leq m$ and $Y^{\sigma}_{[\ell]}=0$ if $\ell$ does not occur in $\J.$

\section{Determinant}

Let $\det:\mat \rightarrow \C$ be the map taking an $n \times n$ complex matrix to its determinant. The Jacobi formula for the derivative of the determinant of a matrix has been well known for a long time. It says that
\begin{equation}
\De \det (A)(X)=\tr (\adj(A)X),\label{jacobi}
\end{equation}
where $\adj(A)$ stands for the \emph{classical adjoint} of $A$. We first note some equivalent descriptions of Jacobi's formula. 

For $1\leq i,j\leq n$ let $A(i|j)$ denote the $(n-1)\times (n-1)$ matrix obtained from $A$ by deleting its $i$th row and $j$th column. Then \eqref{jacobi} can be restated as 
\begin{equation}
\De \det (A)(X)= \sum_{i,j} (-1)^{i+j} x_{ij}\det A(i|j).\label{jacobieq1}
\end{equation}

For $1\leq j\leq n$ let $A(j;X)$ denote the matrix obtained from $A$ by replacing the $j^{th}$ column of $A$ by the $j^{th}$ column of $X$ and keeping the rest of the columns unchanged. Then \eqref{jacobi} can also be written as 
\begin{equation}
\De\det{(A)}(X)= \sum_{j=1}^{n}\det{A(j;X)}. \label{jacobieq2}
\end{equation} 

In \cite{tanvi}, the authors have derived the following formulas for the higher order derivatives of the determinant map that are visible generalisations of \eqref{jacobi}, \eqref{jacobieq1} and \eqref{jacobieq2}.

\begin{thm}
For $1\leq m\leq n$
\begin{equation}
\De^m \det{(A)(X^1,\ldots,X^m)}=\sum_{\sigma \in S_m} \sum_{\J \in Q_{m,n}} \det{A(\J;X^{\sigma(1)},X^{\sigma(2)},\ldots, X^{\sigma(m)})}.\label{derivativedet1}
\end{equation}
In particular, 
\begin{equation}
\De^m \det{(A)(X,\ldots,X)}=m! \sum_{\J \in Q_{m,n}} \det{A(\J;X,\ldots,X)}.
\end{equation}
\end{thm}

Note that for the special case $m=n$ 
\begin{equation}
\De^n \det (A)(X,\ldots,X)=n!\det X.
\end{equation}

To understand the above theorem, first let $n=2$. We know that the determinant function is linear in each of its columns. So it is a bilinear map from $\C^2\times \C^2$ to $\C$ and hence differentiable at every point $(a_1,a_2)\in \C^2\times \C^2$. The derivative is the linear mapping $\De \det(a_1,a_2)$ whose action at any $(x_1,x_2)$ is given by
$$\De \det(a_1,a_2)((x_1,x_2))=\det(a_1,x_2)+\det(x_1,a_2).$$
This is \eqref{derivativedet1} for $n=2$ and $m=1$.
Extending the same idea for any $n$ and any $m$, $1\leq m\leq n$, one can obtain \eqref{derivativedet1}.

Since $\det A$ is an $n$-linear map of its columns, it follows that for $m>n$
\begin{equation}
\De^m \det (A)(X^1,\ldots, X^m)=0.
\end{equation}

\begin{thm}
For $1\leq m\leq n$
\begin{equation}
\De^m \det{(A)(X^1,\ldots,X^m)}=\sum_{\sigma \in S_m} \sum_{\I,\J \in Q_{m,n}} (-1)^{|\I|+|\J|}\det{A(\I|\J)} \det{Y^{\sigma}_{[\J]}[\I|\J]}. \label{derivativedet2}
\end{equation}
In particular, 
\begin{equation}
\De^m \det{(A)(X,\ldots,X)}=m! \sum_{\I,\J \in Q_{m,n}} (-1)^{|\I|+|\J|} \det{A(\I|\J)} \det{X[\I|\J]}.
\end{equation}
\end{thm}

To describe an analogue of Jacobi's formula \eqref{jacobi} we introduce a notation. Let $X^1,\ldots,X^m$ be $m$ operators on $\Hil$. Consider the operator
\begin{equation}
\frac{1}{m!}\sum_{\sigma \in S_m} X^{\sigma(1)}\otimes X^{\sigma(2)}\otimes \cdots \otimes X^{\sigma(m)}\label{operator}
\end{equation}
on the space $\otimes ^m \Hil$. This leaves the space $\wedge^m \Hil$ invariant, and the restriction of this operator to the subspace $\wedge ^m \Hil$ is denoted by $$X^1 \wedge X^2 \wedge \cdots \wedge X^m.$$
The matrix $\adj{A}$ is the transpose of the matrix whose entries are \\ $(-1)^{i+j} \det A(i|j)$. It can be identified with an operator on the space $\wedge^{n-1} \Hil$. Call this operator $\tilde{\wedge}^{n-1} A$. It is unitarily similar to the transpose of the matrix $\wedge^{n-1} A$. Likewise, for $\I,\J \in Q_{m,n}$, the transpose of the matrix with entries $(-1)^{|\I|+|\J|} \det A(\I|\J)$ can be identified with an operator on the space $\wedge^{n-m}\Hil$. Call this operator $\tilde{\wedge}^{n-m} A$. It is unitarily similar to the transpose of the matrix $\wedge^{n-m} A$.
In this notation, the Jacobi formula \eqref{jacobi} can be written as 
\begin{equation}
\De \det (A)(X)=\tr(\tilde{\wedge}^{n-1}A)X.\label{jacobitrform}
\end{equation}

The next theorem is an extension of this.

\begin{thm}
For $1\leq m\leq n$
\begin{equation}
\De^m \det{(A)(X^1,\ldots,X^m)}=m! \tr\left[(\tilde{\wedge}^{n-m}A)(X^1\wedge\cdots\wedge X^m)\right].\label{derivativedet3}
\end{equation}
In particular,
\begin{equation}
\De^m \det{(A)(X,\ldots,X)}=m! \tr\left[(\tilde{\wedge}^{n-m}A)(\wedge^m X)\right].
\end{equation}
\end{thm}

Let $s_1(A)\geq\cdots\geq s_n(A)\geq 0$ be the singular values of $A$ and let $p_k(x_1,\ldots,x_n)$ denote $k$th elementary symmetric polynomial in $n$ variables. From \eqref{jacobitrform} it follows that 
\begin{equation}
\|\De \det A\|=p_{n-1}(s_1(A),\ldots,s_n(A)).\label{detnorm}
\end{equation}
In \cite{friedland}, Bhatia and Friedland proved a more general theorem. They showed that for $1\leq k\leq n$
\begin{equation}
\|\De \wedge^k A\|=p_{k-1}(s_1(A),\ldots,s_k(A)).\label{derantisym}
\end{equation}
For $k=n$ this reduces to \eqref{detnorm}.
Using the theorem stated above, \eqref{detnorm} can be extended to higher order derivatives of the determinant map.

\begin{thm}\label{normdetder}
For $1\leq m\leq n$
\begin{equation}
\|\De^m \det A\|=m!\ p_{n-m}(s_1(A),\ldots,s_n(A)).
\end{equation}
\end{thm}

As a corollary, the following perturbation bound is obtained using Taylor's theorem.

\begin{cor}
Let $X \in \mat$. Then
\begin{equation}
|\det(A+X)-\det(A)|\leq \sum_{m=1}^n p_{n-m}(s_1(A),\ldots,s_n(A)) \|X\|^m.
\end{equation}
Consequently,
\begin{equation}
|\det(A+X)-\det(A)|\leq (\|A\|+\|X\|)^n-\|A\|^n.
\end{equation}
\end{cor}

\section{Permanent}

The \emph{permanent} of $A$, written as $\per{(A)}$, or simply $\per{A}$, is defined by
\begin{equation}
\per{A} = \sum_\sigma a_{1 \sigma (1)}a_{2 \sigma (2)}\cdots a_{n \sigma (n)},\label{1}
\end{equation}
where the summation extends over all the permutations of \{1,\ 2,\ \ldots,\ n\}.
Let $\per: \mat \rightarrow \C$ be the map taking an $n\times n$ matrix to its permanent. This is a differentiable map.

The \emph{permanental adjoint} of $A$, denoted by $\padj{(A)}$, is the $n \times n$ matrix whose $(i,j)$-entry is $\per{A(i|j)}$ (See \cite[p. 237]{merris}). This difference of transpose in the definitions of $\adj$ and $\padj$ is just a matter of convention. We obtain the following result similar to the \emph{Jacobi formula} for determinant.

\begin{thm}
For each $X \in \mat$
\begin{equation}
\De\per{(A)}(X) = \tr(\padj(A)^t X).  \label{4} 
\end{equation}

\end{thm}

This can be restated as 
\begin{equation}
\De\per{(A)}(X)= \sum_{j=1}^{n}\per{A(j;X)}. \label{5}
\end{equation}

The \emph{Laplace expansion theorem} for permanents \cite[p. 16]{minc} says that
for any $1\leq m\leq n$ and for any $\I \in Q_{m,n}$
\begin{equation}
\per{A}=\sum_{\J \in Q_{m,n}} \per{A[\I|\J]} \per{A(\I|\J)}.
\end{equation}
In particular for any $i,\ 1\leq i\leq n$,
\begin{equation}
\per{A}=\sum_{j=1}^{n} a_{ij}\, \per{A(i|j)}.\label{2}
\end{equation} 
Using this, equation \eqref{5} can be rewritten as
\begin{equation}
\De\per{(A)}(X)=\sum_{i=1}^{n} \sum_{j=1}^{n} x_{ij} \per{A(i|j)}. \label{6}
\end{equation}

The following two theorems are analogues of Theorems 2.1 and 2.2 of Section 2 and also generalisations of equations \eqref{5} and \eqref{6} respectively. The key idea here is to use the fact that the permanent function is linear in each of its columns.

\begin{thm}
For $1\leq m\leq n$
\begin{equation}
\De^m \per{(A)(X^1,\ldots,X^m)}=\sum_{\sigma \in S_m} \sum_{\J \in Q_{m,n}} \per{A(\J;X^{\sigma(1)},X^{\sigma(2)},\ldots, X^{\sigma(m)})}.\label{8}
\end{equation}
In particular, 
\begin{equation}
\De^m \per{(A)(X,\ldots,X)}=m! \sum_{\J \in Q_{m,n}} \per{A(\J;X,\ldots,X)}.
\end{equation}
\end{thm}

\begin{thm}
For $1\leq m\leq n$
\begin{equation}
\De^m \per{(A)(X^1,\ldots,X^m)}=\sum_{\sigma \in S_m} \sum_{\I,\J \in Q_{m,n}} \per{A(\I|\J)} \per{Y^{\sigma}_{[\J]}[\I|\J]}. \label{2.9}
\end{equation}
In particular, 
\begin{equation}
\De^m \per{(A)(X,\ldots,X)}=m! \sum_{\I,\J \in Q_{m,n}} \per{A(\I|\J)} \per{X[\I|\J]}.
\end{equation}
\end{thm}

Note that
\begin{equation}
\De^n \per (A)(X,\ldots,X)=n!\per X,
\end{equation}
and for $m>n$
\begin{equation}
\De^m \per (A)(X^1,\ldots, X^m)=0.
\end{equation}

As in the case of determinants, it would be interesting to have an expression analogous to \eqref{4}.
Consider the operator given in \eqref{operator}.
It leaves the space $\vee^m \Hil$ invariant. We use the notation $X^1 \vee X^2 \vee \cdots \vee X^m$ for the restriction of this operator to the subspace $\vee^m \Hil$.

Let $P_m$ be the canonical projection of $\vee^m \Hil$ onto the subspace \\ $\{e_{(\alpha)} : \alpha \in Q_{m,n}\}$. Then there is a permutation of the orthonormal basis $\{m(\alpha)^{-1/2} e_{(\alpha)}: \alpha\in G_{m,n}\}$ in which 
$$P_m=\left[\begin{array}{ccc} 
I &\ O \\
O &\ O
\end{array}\right]$$
and the matrix $T_m=\left(\per A[\alpha|\beta]\right)_{\alpha,\beta\in Q_{m,n}}$ is the upper left corner of $\vee^m A$. Then 
$$P_m (\vee^m A) P_m=
\left[\begin{array}{cc} 
T_m & \ O\\
O  & \ O
\end{array}\right].$$
Let $U$ be the $\binom{n}{m} \times \binom{n}{m}$ unitary matrix given by
$$U=\left[\begin{array}{cccc}
\ & \ & \ &\  1\\
\ & \ &\ 1 & \ \\
\ & \ \text{\rotatebox{90}{\mbox{$\ddots$}}}& \ & \ \\
1& \ & \ & \ 
\end{array}\right].$$
Then $U^* T_m U$ is an $\binom{n}{m} \times \binom{n}{m}$ matrix. For $\alpha,\beta \in Q_{n-m,n}$ the $(\alpha,\beta)$-entry of $U^* T_m U$ is $\per A(\alpha|\beta)$. Let $\tilde{U}$ be the $\binom{n+m-1}{m}\times \binom{n+m-1}{m}$ matrix given by 
$$\tilde{U} =
\left[\begin{array}{cc} 
U & \ O\\
O & \  I
\end{array}\right].$$
We denote by $\tilde {\vee}^m A$, the matrix $\tilde{U}^* (\vee^m A)^t \tilde{U}$. Then 
\begin{equation}
P_m(\tilde {\vee}^m A)P_m=\left[\begin{array}{cc}
U^* T_m^t U & \ O\\
O & \ O
\end{array}\right].
\end{equation}
In particular for $m=n-1$ this becomes
$$P_{n-1}(\tilde {\vee}^{n-1} A)P_{n-1}=\left[\begin{array}{ccc}
(\padj A)^t & \ O\\
O  & \ O
\end{array}\right].$$
Identifying an $n\times n$ matrix $X$ with $\binom{2n-2}{n-1}\times \binom{2n-2}{n-1}$ matrix
$\left[\begin{array}{ccc} 
X & \ O\\
O & \ O\end{array}\right]$, equation \eqref{4} can be written as 

\begin{equation}
\De\per{(A)(X)}=\tr{(P_{n-1}(\tilde{\vee}^{n-1} A)P_{n-1})X}.\label{10}
\end{equation}

Its generalisation for higher order derivatives can be given as follows.

\begin{thm}
For $1\leq m\leq n$
\begin{eqnarray}
\De^m \per{(A)(X^1,\ldots,X^m)}&=&m!\tr\left[\left(P_{n-m}(\tilde{\vee}^{n-m} A)P_{n-m}\right)\right.\nonumber\\
&& \qquad\left.\left(P_m(X^1\vee\cdots\vee X^m)P_m\right)\right].\label{2.17}
\end{eqnarray}
In particular,
$$\De^m \per{(A)(X,\ldots,X)}=m!\tr{\left[\left(P_{n-m}(\tilde{\vee}^{n-m}A)P_{n-m}\right)\left(P_m(\vee^m X)P_m\right)\right]}.$$
\end{thm}

An upper bound for the norms of the higher order derivatives can be obtained from this expression.

\begin{thm}\label{pernorm}
For $1\leq m\leq n$
\begin{equation}
\|\De^m \per A\|\leq \frac{n!}{(n-m)!}\|A\|^{n-m}.
\end{equation}
\end{thm}

By Taylor's theorem, we get the following perturbation bound.

\begin{cor}
Let $X\in \mat$. Then
\begin{equation}
|\per(A+X)-\per A|\leq (\|A\|+\|X\|)^n-\|A\|^n. \label{3.9}
\end{equation}
\end{cor}

\section{Tensor Power}

Let $\otimes^k: \mat \rightarrow \mathbb M(n^k)$ be the map which takes an $n \times n$ matrix $A$ to its $k$th tensor power. Note that for any two matrices $A,B$ 
\begin{equation}
\otimes^k (A+B)=\sum_{\substack{j_i\geq 0\\j_1+\cdots+j_p=k}} (\otimes^{j_1}A) \otimes (\otimes^{j_2}B)\otimes (\otimes^{j_3}A)\otimes \cdots\otimes (\otimes^{j_p}B).\label{expansionoftensor}
\end{equation}

Using the expression for higher order derivatives \eqref{derivativedefn} and the above expansion formula, one can easily see that $\De^m \otimes ^k (A)(X^1,\ldots,X^m)$ is the coefficient of $t_1 t_2\ldots t_m$ in $\otimes^k (A+t_1X^1+\cdots+t_mX^m)$. An explicit expression can be given as follows.

\begin{thm}
For $1 \leq m \leq k$
\begin{eqnarray}
\De^m \otimes ^k (A)(X^1,\ldots,X^m)&=& \nonumber\\
&&\hspace{-3.5cm}\sum_{\sigma \in S_m} \sum_{\substack{j_i \geq 0\\j_1+\cdots+j_{m+1}=k-m}} (\otimes^{j_1}A)\otimes X^{\sigma(1)}(\otimes^{j_2}A)\otimes X^{\sigma(2)}\otimes\cdots \nonumber\\
&&\cdots\otimes (\otimes^{j_m}A) \otimes X^{\sigma(m)}\otimes (\otimes^{j_{m+1}}A).
\end{eqnarray} 
\end{thm}

We note that
\begin{equation}
\De^k \otimes^k(A)(X,\ldots,X)=k!\ (\otimes^k X)
\end{equation}
and for $m>k$
\begin{equation}
\De^m \otimes^k (A)(X^1,\ldots,X^m)=0.
\end{equation}
Norms of these derivatives can be computed from this expression.
\begin{thm}
For $1\leq m\leq k$
\begin{equation}
\|\De^m \otimes ^k (A)\|= \frac{k!}{(k-m)!} \|A\|^{k-m}.\label{tensornorm}
\end{equation}
\end{thm}

The $\leq$ inequality in the above expression is a consequence of the triangle inequality and the fact that $$\|A \otimes B\|= \|A\| \|B\|.$$ The equality in \eqref{tensornorm} is attained at the tuple $\left(A/\|A\|,\ldots,A/\|A\|\right)$.

A perturbation bound follows from here using Taylor's theorem.
\begin{cor}
For $X \in \mat$
\begin{equation}
\|\otimes^k(A+X)-\otimes^k A\| \leq (\|A\|+\|X\|)^k-\|A\|^k.
\end{equation}
\end{cor}

\section{Antisymmetric Tensor Power}
Consider the map $\wedge^k: \mat \rightarrow \mathbb M(\binom{n+k-1}{k})$ which takes an $n \times n$ matrix $A$ to its $k$th antisymmetric tensor power. Recall that the $(\alpha,\beta)$-entry of $\wedge^k A$ is $\det A[\alpha|\beta]$. Using the expression \eqref{derivativedet3}, the higher order derivatives of the map $\wedge^k$ can be obtained. To derive an explicit formula for this, a notation involving multiindices is required.

For elements $ \gamma'=(\gamma'_1,\ldots,\gamma'_m)\in Q_{m,n}$ and $ \alpha=(\alpha_1,\ldots,\alpha_k)\in Q_{k,n}$ we write $\gamma'\subseteq \alpha$ if $1\leq m\leq k\leq n$ and $ \{\gamma'_1,\ldots,\gamma'_m\}\subseteq \{\alpha_1, \ldots, \alpha_k\}$. Further whenever $\gamma'\subseteq \alpha$, we denote by $\alpha-\gamma'$, the element $(\gamma_1,\ldots,\gamma_{k-m})\in Q_{k-m,n}$ where $\{\gamma_1,\ldots,\gamma_{k-m}\}=\{\alpha_1,\ldots,\alpha_k\}\setminus \{\gamma'_1,\ldots,\gamma'_m\}.$
The number $\alpha_1+\cdots+\alpha_k$ is denoted by $|\alpha|$. Let $\alpha'=(1,\ldots,n)-\alpha$. Let $\pi_{\alpha}$ be the permutation on $\{1,2,\ldots,n\}$ defined by $\pi_{\alpha}(\alpha_i)=i$ for all $i=1,\ldots,k$ and $\pi_{\alpha}(\alpha'_j)=k+j$ for all $j=1,\ldots, n-k$.

Let $Y=(y_{\gamma',\delta'})$ be any $\binom{n}{m} \times \binom{n}{m}$ matrix and $\gamma,\delta \in Q_{k-m,n}$. Define an $\binom{n}{k}\times \binom{n}{k}$ matrix 
$Y^{(k)}(\gamma,\delta)$ as follows. For $\alpha, \beta \in Q_{k,n}$ the $(\alpha, \beta)$-entry of $Y^{(k)}(\gamma,\delta)$ is $(-1)^{|\pi_{\alpha}(\gamma)|+|\pi_{\beta}(\delta)|}\ y_{\alpha-\gamma,\beta-\delta}$ if $\gamma \subseteq \alpha$ and $\delta \subseteq \beta$ and $0$ otherwise.

\begin{thm}\label{antisym}
Let $A \in \mat$. Then for $1\leq m\leq k \leq n$
\begin{eqnarray}
\De^m \wedge^k (A)(X^1,\ldots,X^m)&=&\nonumber \\
& &\hspace{-3cm}m!\sum_{\gamma, \delta \in Q_{k-m,n}} \det A[\gamma|\delta] \ (X^1\wedge \cdots \wedge X^m)^{(k)}(\gamma,\delta). 
\end{eqnarray}
In particular,
\begin{equation}
\De^m \wedge^k (A)(X,\ldots,X)=m! \sum_{\gamma, \delta \in Q_{k-m,n}} \det A[\gamma|\delta]\ (\wedge^m X)^{(k)}(\gamma,\delta).
\end{equation}
\end{thm}

Note that 
\begin{equation}
\De^k \wedge^k (A)(X,\ldots,X)=k! \ (\wedge^k X)
\end{equation}
and if $k> n$ or $m>k$, then
\begin{equation}
\De^m \wedge^k (A) (X^1,\ldots,X^m)=0.
\end{equation}
In \cite{friedland}, Bhatia and Friedland gave the norm of the first derivative of the map $\wedge^k$ as follows:
$$\|\De \wedge^k A\|=p_{k-1}(s_1(A),\ldots,s_k(A)).$$ The following theorem by Jain \cite{tanvi2} is an extension of this for its higher order derivatives.

\begin{thm}
For $1\leq m\leq k\leq n$
\begin{equation}
\|\De^m \wedge^k A\|=m!\ p_{k-m}(s_1(A),\ldots,s_k(A)).
\end{equation}
\end{thm}

Note that  for $k=n$ this reduces to Theorem \ref{normdetder} for the determinant map. As a corollary, a perturbation bound can be obtained using Taylor's theorem.

\begin{cor}
For any $ X \in \mat$
\begin{equation}
\|\wedge^k (A+X)-\wedge^k (A)\|\leq \sum_{m=1}^k p_{k-m}(s_1(A),\ldots,s_k(A)) \|X\|^m.
\end{equation}
Consequently,
\begin{equation}
\|\wedge^k (A+X)-\wedge^k (A)\|\leq (\|A\|+\|X\|)^k-\|A\|^k.
\end{equation}
\end{cor}

\section{Symmetric Tensor Power}
Consider the map $\vee^k: \mat \rightarrow \mathbb M(\binom{n+k-1}{k})$ which takes an $n \times n$ matrix $A$ to its $k$th symmetric tensor power. For elements $ \gamma'=(\gamma'_1,\ldots,\gamma'_m)\in G_{m,n}$ and $ \alpha=(\alpha_1,\ldots,\alpha_k)\in G_{k,n}$ we write $\gamma'\subseteq \alpha$ if $1\leq m\leq k$ and $ \{\gamma'_1,\ldots,\gamma'_m\}\subseteq \{\alpha_1, \ldots, \alpha_k\}$, with multiplicities allowed such that if $\alpha_{\ell}$ occurs in $\alpha$, say $d_{\alpha}$ times, then $\alpha_{\ell}$ cannot occur in $\gamma'$ for more than $d_{\alpha}$ times. Also if $\gamma' \subseteq \alpha$, then $\alpha-\gamma'$ will denote the element $(\gamma_1,\ldots,\gamma_{k-m})$ of $G_{k-m,n}$, where $\gamma_{\ell}\in \{\alpha_1,\ldots,\alpha_k\}$ that is, $\gamma_{\ell}$ is some $\alpha_{i_{\ell}}$ and occurs in $\alpha-\gamma'$ exactly $d_{\alpha}-d_{\gamma'}$ times where $d_{\alpha}$ and $d_{\gamma'}$ denote the multiplicities of $\alpha_{i_{\ell}}$ in $\alpha$, and $\gamma'$, respectively.

Let $Y$ be a $\binom{n+m-1}{m} \times \binom{n+m-1}{m}$ matrix and for $1\leq m \leq k$ let $\gamma, \delta \in G_{k-m,n}$. We denote by $Y_{(k)}(\gamma,\delta)$, the $\binom{n+k-1}{k} \times \binom{n+k-1}{k}$ matrix whose indexing set is $G_{k,n}$, and for $\alpha, \beta \in G_{k,n}$ the $(\alpha, \beta)$-entry of $Y_{(k)}(\gamma,\delta)$ is $\left(\frac{m(\alpha-\gamma) m(\beta-\delta)}{m(\alpha) m(\beta)}\right)^{1/2}$ times $(\alpha-\gamma,\beta-\delta)$-entry of $Y$ if $\gamma \subseteq \alpha$ and $\delta \subseteq \beta$ and zero otherwise.

We know that for any $\alpha,\beta\in G_{k,n}$ the $(\alpha,\beta)$-entry of $\vee^k A$ is\\ 
$(m(\alpha) m(\beta))^{-1/2}\per{A[\alpha|\beta]}$. Calculating the derivative of each entry of $\vee^k A$ by using the results from Section 3 will lead to the following.

\begin{thm}
Let $A \in \mat$. Then for $1\leq m\leq k$
\begin{equation}
\De^m \vee^k (A)(X^1,\ldots,X^m)=m!\ \sum_{\gamma, \delta \in G_{k-m,n}} \per A[\gamma|\delta]\ (X^1\vee \cdots \vee X^m)_{(k)}(\gamma,\delta) .\label{3.1}
\end{equation}
\end{thm}

We note here that
\begin{equation}
\De^k \vee^k (A)(X,\ldots,X)=k!\ (\vee^k X)
\end{equation}
and for $m>k$
\begin{equation}
\De^m \vee^k (A)(X^1,\ldots,X^m)=0.
\end{equation}
Bhatia \cite{bhatia} computed the exact norm of the first derivative of the map $\vee^k$:
$$\|\De \vee^k (A)\|=k\|A\|^{k-1}.$$
We extend this result for all order derivatives of the map $\vee^k$ in \cite{grover}.

\begin{thm}
For $1\leq m\leq k$
\begin{equation}
\|\De^m \vee^k A\|=\frac{k!}{(k-m)!}\|A\|^{k-m}.
\end{equation}
\end{thm}

Since $\per A$ is $(\alpha,\alpha)$-entry of $\vee^n A$ for $\alpha=(1,\ldots,n)$, Theorem \ref{pernorm} follows from the above theorem, by putting $k=n$.

By Taylor's theorem, we obtain higher order perturbation bounds.
\begin{cor}
For $X\in \mat$
\begin{equation}
\|\vee^k(A+X)-\vee^k A\|\leq (\|A\|+\|X\|)^k-\|A\|^k.
\end{equation}
\end{cor}

\section{Coefficients of Characteristic Polynomial}

The \emph{characteristic polynomial} of $A$ is defined by
$$\det(xI-A).$$ It can also be written as
\begin{equation}
 x^n - g_1 x^{n-1} + g_2 x^{n-2} - \ldots + (-1)^n g_n \label{charpoly},
\end{equation}
where $g_k$ is the sum of $k \times k$ principal minors of $A$. In particular, $g_1$ is the trace of $A$ and $g_n$ is the determinant of $A$.
We consider $g_k: \mat \rightarrow \C$ as the map taking a matrix to the $k^{th}$ coefficient in \eqref{charpoly}. Then 
\begin{equation}
 g_k(A)=\sum_{\I \in Q_{k,n}} \det A_{\I}, \label{gk1}
 \end{equation}
where $A_{\I}$ denotes the submatrix $A[\I|\I]$ of $A$.
In other words, \eqref{gk1} can also be written as 
\begin{equation}
g_k(A)= \tr (\wedge^k (A)).\label{gk2}
\end{equation}

In \cite{tanvi2} Jain considers the expression \eqref{gk2} to obtain the derivatives of the coefficients $g_k$. For each $k$ the map $g_k$ is a composition of two maps $\wedge^k:\mat\rightarrow \matk$ and the trace map $\tr: \matk\rightarrow \C$. Hence from Theorem \ref{antisym} the derivatives for the coefficients can be obtained. One can also use \eqref{gk1} and the expressions for the derivatives of the determinant to obtain the same expression as given below. We first introduce some notation. 

For $n\times n$ matrices $X^1,\ldots,X^n$ their \emph{mixed discriminant} \cite{bapat2,gurvits} is defined by
$$\Delta(X^1,\ldots,X^n)=\frac{1}{n!}\sum_{\sigma\in S_n} \det\left[X^{\sigma(1)}_{[1]},\ldots, X^{\sigma(n)}_{[n]}\right].$$
The matrix in the square brackets is the matrix whose $j$th column is the $j$th column of $X^{\sigma(j)}$. When all $X^j=X$, $$\Delta(X,\ldots,X)=\det X.$$

\begin{thm}
For $1\leq m\leq k\leq n$
\begin{eqnarray}
\De^m g_k(A)(X^1,\ldots,X^m)&=& m!\sum_{\alpha\in Q_{k,n}}\sum_{\substack{\gamma, \delta \in Q_{k-m,n}\\ \gamma,\delta \subseteq \alpha}} (-1)^{|\pi_{\alpha}(\gamma)|+|\pi_{\alpha}(\delta)|} \det A[\gamma|\delta]\nonumber \\
& & \hspace{-0.5cm}\times \Delta(X^1[\alpha-\gamma|\alpha-\delta],\ldots,X^m[\alpha-\gamma|\alpha-\delta]).\label{dergk}
\end{eqnarray}
In particular,
\begin{eqnarray}
\De^m g_k(A)(X,\ldots,X)&=&\nonumber\\
& &\hspace{-3.4cm}m!\sum_{\alpha\in Q_{k,n}}\sum_{\substack{\gamma, \delta \in Q_{k-m,n}\\ \gamma,\delta \subseteq \alpha}} (-1)^{|\pi_{\alpha}(\gamma)|+|\pi_{\alpha}(\delta)|} \det A[\gamma|\delta] \det X[\alpha-\gamma|\alpha-\delta].\nonumber\\
\end{eqnarray}
\end{thm}

We note here the special case
\begin{equation}
\De^k g_k(A)(X,\ldots,X)=k!\sum_{\alpha\in Q_{k,n}} \det X[\alpha|\alpha]
\end{equation}
and the fact that for $m>k$
\begin{equation}
\De^m g_k(A)(X^1,\ldots, X^m)=0.
\end{equation}
Jain gives another expression for the derivative $\De^m g_k(A)$ in \cite{tanvi2}.
Let $X^1, \ldots, X^m$ be $n\times n $ matrices. Denote the matrix $X^1\wedge\cdots\wedge X^m$ by $\tilde{X}$. For any $\alpha\in Q_{k,n}$ consider the $\binom{n}{k-m}\times \binom{n}{m}$ matrix $Q(\alpha)$ whose $(\gamma,\delta)$-entry $(\gamma\in Q_{k-m,n},\ \delta\in Q_{m,n})$ is given by
$$Q(\alpha)_{\gamma,\delta}=\left\{ \begin{array}{rl}
(-1)^{|\pi_{\alpha}(\gamma)|} &\mbox{ if $\gamma,\delta\subseteq \alpha,\ \gamma\cup\delta=\alpha$}\\
0 &\mbox{ otherwise}
\end{array}
\right.$$
Now define the matrix $\tilde{X}^{(\alpha)}=\left(Q(\alpha)\tilde{X}Q(\alpha)^t\right)^t$, that is, $\tilde{X}^{(\alpha)}$ is the $\binom{n}{k-m}\times \binom{n}{k-m}$ matrix whose $(\gamma,\delta)$-entry is $$(-1)^{|\pi_{\alpha}(\gamma)|+|\pi_{\alpha}(\delta)|}\ \Delta(X^1[\alpha-\gamma|\alpha-\delta],\ldots,X^m[\alpha-\gamma|\alpha-\delta]) \text{ if } \gamma,\delta\subseteq \alpha$$
and is zero otherwise.

Then \eqref{dergk} can also be expressed in the following form.
\begin{thm}
For $1\leq m\leq k\leq n$
\begin{equation}
\De^m g_k(A)(X^1,\ldots,X^m)=m!\tr\left(\wedge^{k-m}(A)\left(\sum_{\alpha\in Q_{k,n}} \tilde{X}^{(\alpha)}\right)\right).
\end{equation}
\end{thm}

Using this Jain \cite{tanvi2} obtains an upper bound on the norms of the derivatives of $g_k$.

\begin{thm}
For $1\leq m\leq k\leq n$
\begin{equation}
\|\De^m g_k(A)\|\leq \frac{(n-k+m)!}{(n-k)!} p_{k-m}(s_1(A),\ldots,s_n(A)).
\end{equation}
\end{thm}

An interesting perturbation bound can be derived using Taylor's theorem.
For nonnegative integers $p,q,r$ with $q+r\leq p$, denote by $\binom{p}{q,r}$, the trinomial coefficient $\frac{p!}{q!\ r! \ (p-q-r)!}$.

\begin{cor}
Let $1\leq k\leq n$. For $X\in \mat$
$$|g_k(A+X)-g_k(A)|\leq \sum_{m=1}^k \binom{n}{k-m,m}\|A\|^{k-m}\|X\|^m.$$
\end{cor}

\textbf{Acknowledgement.}
This article is based on my talk at Indo-French Seminar on Matrix Information Geometries, funded by Indo-French Centre for the Promotion of Advanced Research. I am thankful to my supervisor Prof. Rajendra Bhatia and other participants of the Seminar for their useful comments and suggestions.

\end{document}